\documentclass[10pt, amstex]{article}
\usepackage{amssymb}
\usepackage{fullpage}
\usepackage{amsthm}
\usepackage{amsmath}
\usepackage{hyperref}

\begin{document}
\newcommand{\beq}{\begin{eqnarray}}
\newcommand{\eeq}{\end{eqnarray}}
\newcommand{\beas}{\begin{eqnarray*}}
\newcommand{\enas}{\end{eqnarray*}}
\newcommand{\bea}{\begin{eqnarray}}
\newcommand{\ena}{\end{eqnarray}}
\newcommand{\nn}{\nonumber}
\newcommand{\ignore}[1]{}
\newtheorem{theorem}{Theorem}[section]
\newtheorem{corollary}{Corollary}[section]
\newtheorem{conjecture}{Conjecture}[section]
\newtheorem{proposition}{Proposition}[section]
\newtheorem{remark}{Remark}[section]
\newtheorem{lemma}{Lemma}[section]
\newtheorem{definition}{Definition}[section]
\newtheorem{condition}{Condition}[section]
\newcommand{\pf}{\noindent {\bf Proof:} }
\def\blfootnote{\xdef\@thefnmark{}\@footnotetext}
\title{Concentration of measures via size biased couplings}
\author{Subhankar Ghosh\thanks{Department of Mathematics, University of Southern California, Los Angeles, CA 90089, USA, \texttt{subhankg@usc.edu}}\,\, and Larry Goldstein\thanks{Department of Mathematics, University of Southern California, Los Angeles, CA 90089, USA, \texttt{larry@usc.edu}} \\
\\
{\normalsize{\em University of Southern California } }}

\date{}
\maketitle
 \long\def\symbolfootnote[#1]#2{\begingroup%
\def\thefootnote{\fnsymbol{footnote}}\footnote[#1]{#2}\endgroup}
\symbolfootnote[0]{2000 {\em Mathematics Subject Classification}: Primary 60E15; Secondary 60C05.}
\symbolfootnote[0]{{\em Keywords}: Large deviations, size biased couplings, Stein's method.}


\begin{abstract}
Let $Y$ be a nonnegative random variable with mean $\mu$ and finite positive variance $\sigma^2$, and let $Y^s$, defined on the same space as $Y$, have the $Y$ size biased distribution, that is, the distribution
characterized by
$$
E[Yf(Y)]=\mu E f(Y^s) \quad \mbox{for all functions $f$ for which these expectations exist.}
$$
Under a variety of conditions on the coupling of $Y$ and $Y^s$, including combinations of boundedness and monotonicity, concentration of measure inequalities such
as
\beas
P\left(\frac{Y-\mu}{\sigma}\ge t\right)\le \exp\left(-\frac{t^2}{2(A+Bt)}\right) \quad \mbox{for all $t \ge 0$}
\enas
hold for some explicit $A$ and $B$. Examples include the number of relatively ordered subsequences of a random permutation, sliding window statistics including the number of $m$-runs in a sequence of coin tosses, the number of local maximum of a random function on a lattice, the number
of urns containing exactly one ball in an
urn allocation model, the volume covered by the union of $n$ balls placed uniformly over a volume $n$ subset of $\mathbb{R}^d$, the number of bulbs switched on at the terminal time in the so called lightbulb process, and the infinitely divisible and compound Poisson distributions that satisfy a bounded moment generating function condition.

\end{abstract}
\section{Introduction}
Size biasing random variables is essentially sampling them proportional to their size. Of the many contexts in which size biasing appears, perhaps the most well known is the waiting time paradox, so clearly described in Feller \cite{feller2}, Section I.4. Here, a paradox is generated by the fact that in choosing a time interval `at random' in which to wait for, say buses, it is more likely that an interval with a longer interarrival time is selected. In statistical contexts it has long been known that size biasing may affect a random sample in adverse ways, though at times this same phenomena may also be used to correct for certain biases \cite{midzuno}.

In the realm of normal approximation, size biasing finds a place in Stein's method (see, for instance, \cite{stein},  \cite{ChBa} and \cite{CGS}) alongside the exchangeable pair technique. The areas of application of these two techniques are somewhat complementary, with size biasing useful for the approximation of distributions of nonnegative random variables such as counts, and the exchangeable pair for mean zero variates.  Though Stein's method has been used mostly for assessing the accuracy of normal approximation, recently related ideas have proved to be successful in deriving concentration of measure inequalities, that is, deviation inequalities of the form $P(|Y-E(Y)| \ge t\sqrt{\mbox{Var}(Y)})$, where typically one seeks bounds that decay exponentially in $t$; for a guide to the literature on the concentration of measures, see \cite{ledoux} for a detailed overview. Regarding the use of techniques related to Stein's method to prove such inequalities, Rai\v{c} obtained large deviation bounds for certain graph related statistics in \cite{raic} using the Cram\'{e}r transform and Chatterjee \cite{cha} derived Gaussian and Poisson type tail bounds for Hoeffding's combinatorial CLT and the net magnetization in the Curie-Weiss model in statistical physics in \cite{cha}. While the first paper employs the Stein equation, the later applies constructions which are related to the exchangeable pair in Stein's method (see \cite{Stein86}).

For a given nonnegative random variable $Y$ with finite nonzero mean $\mu$, recall (see \cite{golrinott}, for example) that $Y^s$ has the $Y$-size biased distribution if
\bea \label{EWfWchar}
E[Yf(Y)]=\mu E[f(Y^s)]\quad\mbox{for all functions $f$ for which these expectations exist.}
\ena
Motivated by the complementary connections that exist between the exchangeable pair method and size biasing in Stein's method, we prove the following theorem that shows the parallel persists in the area of concentration of measures, and that size biasing can be used to derive one sided deviation results for nonnegative variables $Y$ that can be closely coupled to a variable $Y^s$ with the $Y$ size biased distribution. Our first result requires the coupling to be bounded. Unbounded couplings are considered in Section \ref{sec:unbounded}, where theorems similar in flavour to the concentration results for the number of isolated vertices in the Erd\H{o}s-R\'{e}nyi in \cite{graph} are derived.

\begin{theorem}
\label{thm:main}
Let $Y$ be a nonnegative random variable with mean and variance $\mu$ and $\sigma^2$ respectively, both finite and positive.  Suppose there exists a coupling of $Y$ to a variable $Y^s$ having the $Y$-size bias distribution which satisfies $|Y^s-Y| \le C$ for some $C>0$ with probability one.
\begin{enumerate}
\item[] If $Y^s\ge Y$ with probability one, then
\bea \label{a}
P\left(\frac{Y-\mu}{\sigma}\le -t\right)\le \exp\left(-\frac{t^2}{2A}\right) \quad \mbox{for all $t > 0$, where $A=C\mu/\sigma^2$.}
\ena

\item[]  If  the moment generating function $m(\theta)=E(e^{\theta Y})$ is finite at $\theta=2/C$, then
\bea \label{b}
P\left(\frac{Y-\mu}{\sigma}\ge t\right)\le \exp\left(-\frac{t^2}{2(A+Bt)}\right)\quad\mbox{for all $t > 0$, where $A=C\mu /\sigma^2$ and $B=C/2\sigma$}.
\ena
\end{enumerate}
\end{theorem}

The monotonicity hypothesis for inequality (\ref{a}), that $Y^s\ge Y$, is natural since $Y^s$ is stochastically larger than $Y$. Therefore there always exists a coupling for which $Y^s\ge Y$. There is no guarantee, however, that for
such a monotone coupling, the difference $Y^s-Y$ is bounded. For (\ref{b}) we note that the moment generating function is
finite everywhere when $Y$ is bounded. In typical examples the variable $Y$ is indexed by $n$, and the ones we consider have the property that the ratio $\mu/\sigma^2$ remains bounded as $n \rightarrow \infty$, and $C$ does not depend on $n$.
In such cases the bound in (\ref{a}) decreases at rate $\exp(-ct^2)$ for some $c>0$, and if $\sigma\rightarrow\infty$ as $n\rightarrow\infty$, the bound in (\ref{b}) is of similar order, asymptotically.

Examples covered by Theorem \ref{thm:main} are given in Section \ref{sec:bndd}, and include the number of relatively ordered subsequences of a random permutation, sliding window statistics including the number of $m$-runs in a sequence of coin tosses, the number of local maximum of a random function on the lattice, the number
of urns containing exactly one ball in the uniform urn allocation model, the volume covered by the union of $n$ balls placed uniformly over a volume $n$ subset of $\mathbb{R}^d$, and the number of bulbs switched on at the terminal time in the so called lightbulb problem.

In Section \ref{sec:unbounded} we also consider some cases where the coupling of $Y^s$ and $Y$ is unbounded, handled on a somewhat case by case basis. Concentration bounds for the number of isolated vertices in the Erd\H{o}s-R\'{e}nyi random graph model, which also falls under this category have been derived in \cite{graph}. In this paper we discuss some further examples including the
infinitely divisible and compound Poisson distributions.  As Theorem \ref{thm:main} shows, additional information is available when the coupling is monotone; this condition holds for the $m$ runs, lightbulb examples, as well as the  infinitely divisible and compound Poisson distributions considered.

A number of results in Stein's method for normal approximation rest on the fact that if a variable $Y$ of interest can be closely coupled to some related variable, then the distribution of $Y$ is close to normal. An advantage, therefore, of the Stein method is that dependence can be handled in a direct manner, by the construction of couplings on the given collection of random variables related to $Y$. In \cite{raic} and \cite{cha}, ideas related to Stein's method were used to obtain concentration of measure inequalities in the presence of dependence.

Of the two, the technique used by Chatterjee in \cite{cha}, based on Stein's exchangeable pair \cite{Stein86}, is the one closer to the approach taken here. We say $Y,Y'$ is a $\lambda$-Stein pair if these variables are exchangeable and satisfy the linearity condition
\bea \label{linlambda}
E(Y-Y'|Y)=\lambda Y \quad\mbox{for some $\lambda\in (0,1)$}.
\ena
The $\lambda$-Stein pair is clearly the special case of the more general identity
\beas
E(F(Y,Y')|Y)=f(Y) \quad \mbox{for some antisymmetric function $F$,}
\enas
specialized to $F(Y,Y')=Y-Y'$ and $f(y)=\lambda y$.
Chatterjee in \cite{cha} considers a pair of variables satisfying this more general identity, and, with
\beas
\Delta(Y)=\frac{1}{2}E((f(Y)-f(Y'))F(Y,Y')|Y),
\enas
obtains a concentration of measure inequality for $Y$ under the assumption that $\Delta(Y)\le Bf(Y)+C$ for some constants $B$ and $C$.

For normal approximation, as seems to be the case here also, the areas in which pair couplings such as (\ref{linlambda}) apply, and those for which size bias coupling of Theorem \ref{thm:main} succeed, appear to be somewhat disjoint. In particular, (\ref{linlambda}) seems to be more suited to variables which arise with mean zero, while the size bias couplings work well for variables, such as counts, which are necessarily nonnegative. Indeed, for the problems we consider, there appears to be no natural way by which to find exchangeable pairs satisfying the conditions of \cite{cha}. On the other hand, the size bias couplings applied here are easy to obtain.

After proving Theorem \ref{thm:main} in Section \ref{sec:main}, in Section \ref{sec:construction} we review the methods in \cite{golrinott} for the construction of size bias couplings in the presence of dependence, and then move to the examples
already mentioned.

\section{Proof of the main result}
\label{sec:main}
In the sequel we make use of the following inequality, which depends on
the convexity of the exponential function;
\bea\label{ineq-main}
\frac{e^y-e^x}{y-x}=\int_0^1e^{ty+(1-t)x}dt \le \int_0^1(te^y+(1-t)e^x)dt =\frac{e^y+e^x}{2} \quad \mbox{for all $x \not = y$.}
\ena
We now move to the proof of Theorem \ref{thm:main}.
\begin{proof}
Recall $Y^s$ is given on the same space as $Y$, and has the $Y$ size biased distribution. By (\ref{ineq-main}), for all $\theta \in \mathbb{R}$, since $|Y^s-Y| \le C$,
\bea\label{ineq-exp-gen}
|e^{\theta Y^s}-e^{\theta Y}| \le \frac{1}{2}|\theta(Y^s-Y)|(e^{\theta Y^s}+e^{\theta Y})\le \frac{C|\theta|}{2}(e^{\theta Y^s}+e^{\theta Y}).
\ena
Recalling that if the moment generating function $m(\theta)=E[e^{\theta Y}]$ exists in an open interval containing $\theta$ then we may differentiate under the expectation, we obtain
\bea \label{diffunder}
m'(\theta)=E[Ye^{\theta Y}]= \mu E[e^{\theta Y^s}].
\ena

To prove (\ref{a}), let $\theta<0$ and note that since the coupling is monotone $\exp(\theta Y^s)\le \exp(\theta Y)$.
Now (\ref{ineq-exp-gen}) yields
\beas
e^{\theta Y}-e^{\theta Y^s}\le {C|\theta|}e^{\theta Y}.
\enas
Since $Y \ge 0$ the moment generating function $m(\theta)$ exists for all $\theta < 0$, so taking expectation and rearranging yields
\beas
Ee^{\theta Y^s} \ge (1-C|\theta|)Ee^{\theta Y}=(1+C\theta)E(e^{\theta Y}),
\enas
and now, by (\ref{diffunder}),
\bea\label{control-left}
m'(\theta)\ge \mu(1+C\theta)m(\theta)\quad\mbox{for all $\theta<0$}.
\ena
To consider standardized deviations of $Y$, that is, deviations of $|Y-\mu|/\sigma$, let
\bea
M(\theta)=Ee^{\theta(Y-\mu)/\sigma}=e^{-\theta \mu/\sigma}m(\theta/\sigma).
\label{def-M-theta}
\ena
Now rewriting (\ref{control-left}) in terms of $M(\theta)$, we obtain for all $\theta<0$,
\bea
M'(\theta) &=& -(\mu/\sigma) e^{-\theta \mu/\sigma}m(\theta/\sigma) +  e^{-\theta \mu/\sigma}m'(\theta/\sigma)/\sigma\nonumber\\
           &\ge& -(\mu/\sigma) e^{-\theta \mu/ \sigma}m(\theta/\sigma) + (\mu/\sigma) e^{-\theta \mu/\sigma}\left( 1+\frac{C\theta}{\sigma} \right) m(\theta/\sigma)\nonumber\\
           &=& (\mu/\sigma^2) C\theta M(\theta)\label{bd-growth-neg-M}.
\ena

Since $M(0)=1$, by (\ref{bd-growth-neg-M})
\beas
-\log M(\theta) = \int_\theta^0 \frac{M'(s)}{M(s)}ds \ge \int_\theta^0 \frac{C\mu s}{\sigma^2}ds = -\frac{C \mu \theta^2}{2 \sigma^2},
\enas
so exponentiation gives us
\beas
M(\theta)\le \exp\left(\frac{C\mu\theta^2}{2\sigma^2}\right)\quad\mbox{when $\theta<0$}.
\enas
Hence for a fixed $t > 0$, for all $\theta<0$,
\bea
\nn P\left(\frac{Y-\mu}{\sigma}\le -t\right) = P\left(\theta \left( \frac{Y-\mu}{\sigma} \right) \ge -\theta t\right)
&=& P\left(e^{\theta \left( \frac{Y-\mu}{\sigma} \right)} \ge e^{-\theta t} \right)\\
&\le& e^{\theta t}M(\theta)\le \exp\left(\theta t+\frac{C\mu\theta^2}{2\sigma^2}\right).\label{bd-conc-neg}
\ena
Substituting $\theta=-t\sigma^2/(C\mu)$ into (\ref{bd-conc-neg}) completes the proof of (\ref{a}).

Moving on to the proof of (\ref{b}),  taking expectation in (\ref{ineq-exp-gen}) with $\theta>0$, we obtain
\beas
Ee^{\theta Y^s}-Ee^{\theta Y} \le \frac{C\theta}{2}\left( Ee^{\theta Y^s}+Ee^{\theta Y}\right),
\enas
so in particular, when $0<\theta < 2/C$,
\bea\label{exp-size-bias}
E[e^{\theta Y^s}] \le \left( \frac{1+C \theta/2}{1-C \theta/2} \right) E[e^{\theta Y}].
\ena
As $m(2/C)<\infty$, (\ref{diffunder}) applies and (\ref{exp-size-bias}) yields
\bea \label{mprimebound}
m'(\theta) \le \mu \left( \frac{1+C \theta/2}{1-C \theta/2} \right) m(\theta)\quad\mbox{for all $0<\theta<2/C$}.
\ena
Now letting $\theta \in (0,2\sigma/C)$, from (\ref{def-M-theta}), $M(\theta)$ is differentiable for all $\theta < 2\sigma/C$ and
(\ref{mprimebound}) yields,
\beas
M'(\theta) &=& -(\mu/\sigma) e^{-\theta \mu/\sigma}m(\theta/\sigma) +  e^{-\theta \mu/\sigma}m'(\theta/\sigma)/\sigma\\
           &\le& -(\mu/\sigma) e^{-\theta \mu/\sigma}m(\theta/\sigma) + (\mu/\sigma) e^{-\theta \mu/\sigma}\left( \frac{1+C \theta/(2\sigma)}{1-C \theta/(2\sigma)} \right) m(\theta/\sigma)\\
           &=& (\mu/\sigma) e^{-\theta \mu/\sigma}m(\theta/\sigma) \left(\left( \frac{1+C \theta/(2\sigma)}{1-C \theta/(2\sigma)} \right)-1\right)\\
           &=& (\mu/\sigma^2) \left( \frac{C\theta}{1-C\theta/(2\sigma)}\right)M(\theta).
\enas

Dividing by $M(\theta)$ we may rewrite the inequality as
\beas
\frac{d}{d\theta}\log M(\theta)\le (\mu/\sigma^2) \left( \frac{C\theta}{1-C\theta/(2\sigma)}\right).
\enas
Noting that $M(0)=1$, setting $A=C\mu/\sigma^2$ and $B=C/(2\sigma)$, integrating we obtain
\beas
\log M(\theta) &=& \int_0^\theta \frac{d}{ds}\log M(s)\, ds \le (\mu/\sigma^2)  \int_0^\theta  \left( \frac{Cs}{1-B\theta} \right)ds
= (\mu/\sigma^2 ) \frac{C\theta^2}{2(1-B\theta)}
=\frac{A\theta^2}{2(1-B\theta)}.
\enas
Hence, for $t>0$,
\beas
P\left(\frac{Y-\mu}{\sigma}\ge t\right)=P\left(\theta ( \frac{Y-\mu}{\sigma}) \ge\theta t\right)=P\left(e^{\theta\left(\frac{Y-\mu}{\sigma}\right)}\ge e^{\theta t}\right)\le e^{-\theta t}M(\theta)\le e^{-\theta t}\exp\left(\frac{A\theta^2}{2(1-B\theta)}\right).
\enas
Noting that $\theta=t/(A+Bt)$ lies in $(0,2\sigma/C)$ for all $t>0$, substituting
this value yields the bound
\beas
P\left(\frac{Y-\mu}{\sigma}\ge t\right)< \exp\left(-\frac{t^2}{2(A+Bt)}\right) \quad \mbox{for all $t>0$,}
\enas
completing the proof.
\end{proof}

\section{Construction of size bias couplings}
\label{sec:construction}
In this section we will review the discussion in \cite{golrinott} which gives a procedure for a construction of size bias couplings when $Y$ is a sum; the method has its roots in the work of Baldi et al. \cite{baldi}.
The construction depends on being able to size bias a collection of nonnegative random variables in a given coordinate, as described in the following definition.
Letting $F$ be the distribution of $Y$, first note that the characterization (\ref{EWfWchar}) of the size bias distribution $F^s$ is equivalent
to the specification of $F^s$ by its Radon Nikodym derivative
\bea
dF^s(x)=\frac{x}{\mu}dF(x).\label{sizebias-df}
\ena
\begin{definition}
\label{sb-coordinate-i} Let ${\cal A}$ be an arbitrary index set and let $\{X_\alpha:\alpha\in {\cal A}\}$ be a collection of nonnegative random variables with finite, nonzero expectations $EX_\alpha=\mu_\alpha$ and joint distribution $dF(\mathbf{x})$. For $\beta\in{\cal A}$, we say that $\mathbf{X}^\beta=\{X^\beta_\alpha:\alpha\in {\cal A}\}$ has the $\mathbf{X}$ size bias distribution in coordinate $\beta$ if $\mathbf{X}^\beta$ has joint distribution
$$
dF^\beta(\mathbf{x})=x_\beta dF(\mathbf{x})/\mu_\beta.
$$
\end{definition}
Just as (\ref{sizebias-df}) is related to (\ref{EWfWchar}), the random vector $\mathbf{X}^\beta$ has the $\mathbf{X}$ size bias distribution in coordinate $\beta$ if and only if
\beas
E[X_\beta f(\mathbf{X})]=\mu_\beta E[f(\mathbf{X}^\beta)] \quad \mbox{for all functions $f$ for which these expectations exist.}
\enas
Now letting $f(\mathbf{X})=g(X_\beta)$ for some function $g$ one recovers (\ref{EWfWchar}),
showing that the $\beta^{th}$ coordinate of $\mathbf{X}^\beta$, that is, $X^\beta_\beta$, has the $X_\beta$ size bias distribution.

The factorization
$$
P(\mathbf{X} \in d \mathbf{x})=P(\mathbf{X} \in d\mathbf{x}|X_\beta=x) P(X_\beta\in dx)
$$
of the joint distribution of $\mathbf{X}$ suggests a way to construct $\mathbf{X}$. First generate $X_\beta$, a variable with distribution $P(X_\beta\in dx)$. If $X_\beta=x$, then generate the remaining
variates $\{X^\beta_\alpha,\alpha \not = \beta\}$ with distribution $P(\mathbf{X}\in d\mathbf{x}|X_\beta=x)$.
Now, by the factorization of $dF({\bf x})$, we have
\bea \label{dFsfactors}
dF^\beta(\mathbf{x})=x_\beta dF(\mathbf{x})/\mu_\beta =P(\mathbf{X} \in d\mathbf{x}|X_\beta=x) x_\beta P(X_\beta\in dx)/\mu_\beta = P(\mathbf{X} \in d\mathbf{x}|X_\beta=x) P(X^\beta_\beta\in dx).
\ena
Hence, to generate $\mathbf{X}^\beta$ with distribution $dF^\beta$, first generate a variable $X^\beta_\beta$ with the
$X_\beta$ size bias distribution, then, when $X_\beta^\beta=x$, generate the remaining variables according to their original conditional distribution given that the $\beta^{th}$ coordinate takes on the value $x$.

Definition \ref{sb-coordinate-i} and the following proposition from Section 2 of \cite{golrinott} will be applied
in the subsequent constructions; the reader is referred there for the simple proof.

\begin{proposition}
\label{prop-golrinott}
Let ${\cal A}$ be an arbitrary index set, and let ${\bf X}=\{X_{\alpha}, \alpha \in {\cal A}\}$ be a collection of nonnegative random variables with finite means. For any subset $B \subset {\cal A}$, set
$$
X_B = \sum_{\beta \in  B}X_{\beta} \quad \mbox{and} \quad \mu_B=EX_B.
$$
Suppose $B \subset {\cal A}$ with $0<\mu_B< \infty$,
and for $\beta \in B$  let ${\bf X}^{\beta}$ have the ${\bf X}$-size biased distribution in coordinate $\beta$ as in Definition \ref{sb-coordinate-i}.
If ${\bf X}^B$ has the mixture distribution
$$
{\cal L}({\bf X}^B)=\sum_{\beta \in B} \frac{\mu_\beta}{\mu_B}{\cal L}({\bf X}^\beta),
$$
then
\beas
EX_B f({\bf X})=\mu_B Ef({\bf X}^B)
\enas
for all real valued functions $f$ for which these expectations exist.
Hence, for any $A \subset {\cal A}$, if $f$ is
a function of $X_A=\sum_{\alpha \in A}X_\alpha$ only,
\begin{equation}
\label{notasn}
EX_B f(X_A)=\mu_B Ef(X_A^B) \quad \mbox{where} \quad X_A^B= \sum_{\alpha \in A} X_{\alpha} ^B.
\end{equation}
Taking $A=B$ in (\ref{notasn}) we have $EX_Af(X_A)=\mu_A Ef(X_A^A)$,
and hence $X_A^A$ has the $X_A$-size biased distribution, as in (\ref{EWfWchar}).
\end{proposition}

In our examples we use Proposition \ref{prop-golrinott} and (\ref{dFsfactors}) to obtain a variable $Y^s$ with the size bias distribution of $Y$, where $Y=\sum_{\alpha \in A} X_\alpha$, as follows. First choose a random index $I \in A$ with probability
$$
P(I=\alpha)=\mu_\alpha/\mu_A,  \quad \mbox{$\alpha \in A$.}
$$
Next generate $X^I_I$ with the size bias distribution of $X_I$. If $I=\alpha$ and $X_\alpha^\alpha=x$, generating $\{X_\beta^\alpha: \beta \in A \setminus \{\alpha\}\}$ using the (original) conditional distribution
\beas
P(X_\beta, \beta \not = \alpha | X_\alpha=x),
\enas
the sum $Y^s=\sum_{\alpha \in A} X_\alpha^I$ has the $Y$ size biased distribution.

\section{Applications: bounded couplings}
\label{sec:bndd}
We now consider the application of Theorem \ref{thm:main} to derive concentration of measure results for the number of relatively ordered subsequences of a random permutation, the number of $m$-runs in a sequence of coin tosses, the number of local extrema on a graph, the number of nonisolated balls in an urn allocation model, the covered volume in binomial coverage process, and the number of bulbs lit at the terminal time in the so called lightbulb process. Without further mention we will use the fact that when (\ref{a}) and (\ref{b}) hold for some $A$ and $B$ then they also hold when these values are replaced by any larger ones, which may also be denoted by $A$ and $B$.

\subsection{Relatively ordered sub-sequences of a random
permutation}
For $n \ge m \ge 3$, let $\pi$ and $\tau$ be
permutations of ${\cal V}=\{1,\ldots,n\}$ and $\{1,\ldots,m\}$, respectively,
and let
$$
{\cal V}_\alpha = \{\alpha,\alpha+1,\ldots,\alpha+m-1\} \quad \mbox{for $\alpha \in {\cal V}$,}
$$
where addition of elements of ${\cal V}$ is modulo $n$.
We say the pattern $\tau$ appears at location $\alpha \in {\cal V}$ if
the values $\{\pi(v)\}_{v \in {\cal V}_\alpha}$ and
$\{\tau(v)\}_{v \in {\cal V}_1}$ are in the same relative order.
Equivalently, the pattern $\tau$ appears at $\alpha$ if
and only if $\pi(\tau^{-1}(v)+\alpha-1), v  \in {\cal V}_1$ is an
increasing sequence. When $\tau=\iota_m$, the identity permutation of length $m$, we
say that $\pi$ has a rising sequence of length $m$ at position
$\alpha$. Rising sequences are studied in
\cite{BaDia} in connection with card tricks and card shuffling.

Letting $\pi$ be chosen uniformly from all permutations of $\{1,\ldots,n\}$, and
$X_\alpha$ the indicator
that $\tau$ appears at $\alpha$,
$$
X_\alpha(\pi(v), v \in {\cal V}_\alpha) = 1(\pi(\tau^{-1}(1)+\alpha -1 ) < \cdots <
\pi(\tau^{-1}(m)+\alpha-1)),
$$
the sum $Y=\sum_{\alpha \in {\cal V}}
X_\alpha$ counts the number of $m$-element-long segments of $\pi$
that have the same relative order as $\tau$.

For $\alpha \in {\cal V}$ we may generate ${\bf
X}^\alpha=\{X_\beta^\alpha, \beta \in {\cal V}\}$ with the ${\bf X}=\{X_\beta, \beta \in {\cal V}\}$
distribution size biased in direction $\alpha$, following \cite{goldstein1}.
Let
$\sigma_\alpha$ be the permutation of $\{1,\ldots,m\}$ for which
$$
\pi(\sigma_\alpha(1)+\alpha-1)<\cdots < \pi(\sigma_\alpha(m)+\alpha-1),
$$
and set
\beas
\pi^\alpha(v)=\left\{
\begin{array}{cl}
\pi(\sigma_\alpha(\tau(v-\alpha+1))+\alpha-1), & v \in {\cal V}_\alpha\\
\pi(v) & v \not \in {\cal V}_\alpha.
\end{array}
\right.
\enas
In other words $\pi^\alpha$ is the permutation
$\pi$ with the values $\pi(v), v \in {\cal V}_\alpha$ reordered so that
$\pi^\alpha(\gamma)$ for $\gamma \in {\cal V}_\alpha$ are
in the same relative order as $\tau$. Now let
$$
X_\beta^\alpha=X_\beta(\pi^\alpha(v), v \in {\cal V}_\beta),
$$
the indicator that $\tau$ appears at position $\beta$ in
the reordered permutation $\pi^\alpha$. As $\pi^\alpha$ and $\pi$ agree except
perhaps for the $m$ values in ${\cal V}_\alpha$, we have
$$
X_\beta^\alpha=X_\beta(\pi(v), v\in {\cal V}_\beta) \quad \mbox{for all $|\beta-\alpha| \ge m$.}
$$
Hence, as
\beas 
|Y^\alpha-Y| \le \sum_{|\beta-\alpha|\le m-1} |X_\beta^\alpha - X_\beta| \le 2m-1.
\enas
we may take $C=2m-1$ as the almost sure bound on the coupling of $Y^s$ and $Y$.

Regarding the mean $\mu$ of $Y$, clearly for any $\tau$, as all relative orders of $\pi(v), v \in {\cal V}_\alpha$ are equally likely,
\beas 
EX_\alpha=1/m! \quad \mbox{and therefore} \quad \mu=n/m!.
\enas
To compute the variance, for $0 \le k \le m-1$, let $I_k$ be the indicator that $\tau(1),\ldots,\tau(m-k)$ and $\tau(k+1),\ldots,\tau(m)$ are in the same relative order. Clearly $I_0=1$, and for rising sequences,
as $\tau(j)=j$, $I_k=1$ for all $k$. In general for $0 \le k \le m-1$ we have $X_\alpha X_{\alpha+k}=0$
if $I_k=0$, as the joint event in this case demands two different relative orders on the segment of $\pi$
of length $m-k$ of which both $X_\alpha$ and $X_{\alpha+k}$ are a function. If $I_k=1$ then a given, common,
relative order is demanded for this same length of $\pi$, and relative orders also for the two segments of length $k$ on which exactly one of $X_\alpha$ and $X_\beta$ depend, and so, in total a relative order on $m-k+2k=m+k$ values of $\pi$,
and therefore
\beas
EX_\alpha X_{\alpha+k}=I_k/(m+k)! \quad \mbox{and} \quad \mbox{Cov}(X_\alpha,X_{\alpha+k})=I_k/(m+k)!-1/(m!)^2.
\enas
As the relative orders of non-overlapping segments of $\pi$ are independent, now taking $n \ge 2m$, the variance $\sigma^2$ of $Y$ is given by
\beas
\sigma^2 &=& \sum_{\alpha \in {\cal V}}\mbox{Var}(X_\alpha) + \sum_{\alpha \not = \beta}\mbox{Cov}(X_\alpha,X_\beta)\\
&=& \sum_{\alpha \in {\cal V}}\mbox{Var}(X_\alpha) + \sum_{\alpha \in {\cal V}}\sum_{\beta:1 \le |\alpha - \beta| \le m-1}\mbox{Cov}(X_\alpha,X_\beta)\\
&=& \sum_{\alpha \in {\cal V}}\mbox{Var}(X_\alpha) + 2\sum_{\alpha \in {\cal V}} \sum_{k=1}^{m-1}\mbox{Cov}(X_\alpha,X_{\alpha+k})\\
&=& n\mbox{Var}(X_1) + 2n \sum_{k=1}^{m-1}\mbox{Cov}(X_1,X_{1+k})\\
&=& n\left(\frac{1}{m!}-\frac{1}{(m!)^2}\right) + 2n \sum_{k=1}^{m-1}\left( \frac{I_k}{(m+k)!}-(\frac{1}{m!})^2\right)\\
&=& n\left( \frac{1}{m!}\left(1-\frac{2m-1}{m!}\right)+2 \sum_{k=1}^{m-1} \frac{I_k}{(m+k)!} \right).
\enas
Clearly $\mbox{Var}(Y)$ is maximized for the identity permutation $\tau(k)=k,k=1,\ldots,m$, as $I_m=1$ for all $1 \le m \le m-1$, and as mentioned, this case corresponds to counting the number of rising sequences. In contrast, the variance lower bound
\beas 
\sigma^2 \ge \frac{n}{m!}\left(1-\frac{2m-1}{m!}\right)
\enas
is attained at the permutation
\beas
\tau(j) =\left\{
\begin{array}{cl}
1 & j=1\\
j+1 & 2 \le j \le m-1\\
2 & j=m
\end{array}
\right.
\enas
which has $I_k=0$ for all $1 \le k \le m-1$.
In particular, the bound (\ref{b}) of Theorem \ref{thm:main} holds with
\beas
A=\frac{2m-1}{1-\frac{2m-1}{m!}} \quad \mbox{and} \quad B=\frac{2m-1}{2\sqrt{\frac{n}{m!}\left(1-\frac{2m-1}{m!}\right)}}.
\enas

\subsection{Local Dependence}
The following lemma shows how to construct a collection of variables ${\bf X}^\alpha$ having the ${\bf X}$ distribution biased in direction $\alpha$ when $X_\alpha$ is some function of a subset of a collection of independent random variables.
\begin{lemma}
\label{lem:loc}
Let $\{C_g, g \in {\cal V}\}$ be a collection of independent random variables, and for each $\alpha \in {\cal V}$ let ${\cal V}_\alpha \subset {\cal V}$ and $X_\alpha=X_\alpha(C_g,g \in {\cal V}_\alpha)$ be a nonnegative random variable with a nonzero, finite expectation. Then if $\{C_g^\alpha, g \in {\cal V}_\alpha\}$ has distribution
$$
dF^\alpha(c_g,g \in {\cal V}_\alpha)=
\frac{X_\alpha(c_g,g \in {\cal V}_\alpha)}{EX_\alpha(C_g,g \in {\cal V}_\alpha)}dF(c_g,g\in {\cal V}_\alpha)
$$
and is independent of $\{C_g, g \in {\cal V}\}$, letting
\beas
X_\beta^\alpha=X_\beta(C_g^\alpha, g \in {\cal V}_\beta \cap {\cal V}_\alpha, C_g, g \in {\cal V}_\beta \cap {\cal V}_\alpha^c),
\enas
the collection ${\bf X}^\alpha=\{X_\beta^\alpha, \beta \in {\cal V}\}$ has the ${\bf X}$ distribution biased in direction $\alpha$.

Furthermore, with $I$ chosen proportional to $EX_\alpha$, independent of the remaining variables, the sum
$$
Y^s=\sum_{\beta \in {\cal V}}X_\beta^I
$$
has the $Y$ size biased distribution, and when there exists $M$ such that $X_\alpha \le M$ for all $\alpha$,
\bea \label{def:b}
|Y^s-Y| \le bM \quad \mbox{where} \quad b=\max_\alpha |\{\beta: {\cal V}_\beta \cap {\cal V}_\alpha \not = \emptyset\}|.
\ena
\end{lemma}

\begin{proof} By independence, the random variables
\beas
\{C_g^\alpha, g \in {\cal V}_\alpha\} \cup  \{C_g, g \not \in {\cal V}_\alpha\}
\quad
\mbox{have distribution}
\quad
dF^{\alpha}(c_g,g \in {\cal V}_\alpha) dF(c_g, g \not \in {\cal V}_\alpha).
\enas
Thus, with ${\bf X}^\alpha$ as given, we find
\beas
EX_\alpha f({\bf X}) &=& \int x_\alpha f({\bf x})dF(c_g,g \in {\cal V})\\
&=& EX_\alpha \int  f({\bf x}) \frac{x_\alpha dF(c_g,g \in {\cal V}_\alpha)}{EX_\alpha(C_g,g \in {\cal V}_\alpha)}dF(c_g, g\not \in {\cal V}_\alpha)\\
&=& EX_\alpha \int  f({\bf x}) dF^\alpha(c_g,g \in {\cal V}_\alpha)dF(c_g, g\not \in {\cal V}_\alpha)\\
&=& EX_\alpha Ef({\bf X}^\alpha).
\enas
That is, ${\bf X}^\alpha$ has the ${\bf X}$ distribution biased in direction $\alpha$, as in Definition \ref{sb-coordinate-i}.

The claim on $Y^s$ follows from Proposition \ref{prop-golrinott}, and finally, since $X_\beta=X_\beta^\alpha$ whenever ${\cal V}_\beta \cap {\cal V}_\alpha = \emptyset$,
$$
|Y^s-Y| \le \sum_{\beta: {\cal V}_\beta \cap {\cal V}_I \not = \emptyset}|X_\beta^I-X_\beta| \le bM.
$$
This completes the proof.
\end{proof}

\subsubsection{Sliding $m$ window statistics}
For $n \ge m
\ge 1$, let ${\cal V}=\{1,\ldots,n\}$ considered modulo
$n$, $\{C_g:g \in {\cal V}\}$ i.i.d. real valued random variables,
and for each $\alpha \in {\cal V}$ set
\beas
{\cal V}_\alpha=\{v \in {\cal V}: \alpha \le v \le \alpha +
m -1 \}.
\enas
Then for $X:\mathbb{R}^m \rightarrow [0,1]$, say, Lemma
\ref{lem:loc} may be applied to the sum $Y=\sum_{\alpha
\in {\cal V}}X_\alpha$ of the $m$-dependent sequence $X_\alpha =
X(C_\alpha, \ldots, C_{\alpha+m-1})$, formed by applying the
function $X$ to the variables in the `$m$-window' ${\cal
V}_\alpha$. As for all $\alpha$ we have $X_\alpha \le 1$ and
$$
\max_\alpha |\{\beta: {\cal V}_\beta \cap {\cal V}_\alpha \not = \emptyset\}|=2m-1,
$$
we may take $C=2m-1$ in Theorem \ref{thm:main}, by Lemma \ref{lem:loc}.

For a concrete example let
$Y$ be the number of $m$ runs of the sequence $\xi_1,\xi_2,\ldots,\xi_n$ of $n$ i.i.d Bernoulli($p$) random variables with $p\in(0,1)$, given by $Y=\sum_{i=1}^n X_i$ where $X_i=\xi_i\xi_{i+1}\cdots\xi_{i+m-1}$, with the periodic convention
$\xi_{n+k}=\xi_k$. In \cite{multi}, the authors develop smooth function bounds for normal approximation for the case of 2-runs.
Note that the construction given in Lemma \ref{lem:loc} for this case is monotone, as for any $i$, letting
\beas
\xi_j'=\left\{
\begin{array}{cl}
\xi_j & j \not \in \{i,\ldots,i+m-1\}\\
1     & j \in \{i,\ldots,i+m-1\},
\end{array}
\right.
\enas
the number of $m$ runs of $\{\xi_j'\}_{i=1}^n$, that is $Y^s=\sum_{i=1}^n \xi_i'\xi_{i+1}'\cdots \xi_{i+m-1}'$, is at least $Y$.

For the mean of $Y$ clearly $\mu=np^m$. For the variance, now letting $n \ge 2m$ and using the fact that non-overlapping segments of the sequence are independent,
\beas
\sigma^2&=&\sum_{i=1}^n \mbox{Var}(\xi_{i}\xi_{i+1}\cdots \xi_{i+m-1})+2\sum_{i<j}\mbox{Cov}(\xi_i\cdots \xi_{i+m-1},\xi_{j}\cdots \xi_{j+m-1})\nonumber\\
&=&np^m(1-p^m)+2\sum_{i=1}^n \sum_{j=1}^{m-1}\mbox{Cov}(\xi_i\cdots \xi_{i+m-1},\xi_{i+j}\cdots \xi_{i+j+m-1}).
\enas
For the covariances,
\beas
\mbox{Cov}(\xi_i\cdots \xi_{i+m-1},\xi_{i+j}\cdots \xi_{i+j+m-1})&=&E(\xi_i\cdots \xi_{i+j-1}\xi_{i+j}\cdots \xi_{i+m-1} \xi_{i+m}\cdots \xi_{i+j+m-1})-p^{2m}\\
&=&p^{m+j}-p^{2m}, 
\enas
and therefore
\beas
\sigma^2
=np^m\left( (1-p^m)+2\left(\frac{p-p^m}{1-p}-(m-1)p^m\right)\right)
=np^m\left( 1 + 2 \frac{p-p^m}{1-p} - (2m-1)p^m\right).
\enas
Hence (\ref{a}) and (\ref{b}) of Theorem \ref{thm:main} hold with
\beas
A=\frac{2m-1}{ 1 + 2 \frac{p-p^m}{1-p} - (2m-1)p^m} \quad \mbox{and} \quad B=\frac{2m-1}{2\sqrt{np^m\left( 1 + 2 \frac{p-p^m}{1-p} - (2m-1)p^m\right)}}.
\enas

\subsubsection{Local extrema on a lattice}

Size biasing the number of local extrema on graphs, for the purpose of normal approximation, was studied in
\cite{baldi} and \cite{goldstein1}. For a given graph ${\cal G}=\{{\cal V},{\cal E}\}$, let ${\cal G}_v=\{{\cal V}_v,{\cal E}_v\}, v \in {\cal V}$,
be a collection of isomorphic subgraphs of ${\cal G}$ such that
$v \in {\cal V}_v$ and for all $v_1,v_2 \in {\cal V}$ the isomorphism from  ${\cal G}_{v_1}$ to ${\cal G}_{v_2}$ maps $v_1$ to $v_2$.
Let $\{C_g, g \in {\cal V} \}$ be a collection of independent and
identically distributed random variables, and let $X_v$ be
defined by
$$
X_v(C_w, w \in {\cal V}_v) = 1(C_v > C_w, w
\in {\cal V}_v), \quad v \in {\cal V}.
$$
Then the sum $Y=\sum_{v \in {\cal V}} X_v$ counts the
number local maxima. In general one may define the neighbor distance $d$ between two vertices $v,w \in {\cal V}$ by
$$
d(v,w)=\min\{n:\mbox{there $\exists\,\,v_0,\ldots,v_n$ in ${\cal V}$ such that $v_0=v, v_n=w$ and $(v_k,v_{k+1}) \in {\cal E}$ for $k=0,\ldots,n$} \}.
$$
Then for $v \in {\cal V}$ and $r=0,1,\ldots$,
\beas
{\cal V}_v(r) =\{w \in {\cal V}: d(w,v) \le r\}
\enas
is the set of vertices of ${\cal V}$ at distance at most $r$ from $v$.
We suppose that the given isomorphic graphs are of this form, that is, that there is some $r$ such that ${\cal V}_v={\cal V}_v(r)$ for all $v \in {\cal V}$. Then if $d(v_1,v_2)>2r$, and $(w_1,w_2) \in {\cal V}_{v_1} \times {\cal V}_{v_2}$, rearranging
\beas
2r<d(v_1,v_2)\le d(v_1,w_1)+d(w_1,w_2)+d(w_2,v_2)
\enas
and using $d(v_i,w_i) \le r, i=1,2,$ yields $d(w_1,w_2)>0$. Hence,
\bea \label{2r-1}
d(v_1,v_2)>2r \quad \mbox{implies} \quad {\cal V}_{v_1} \bigcap {\cal V}_{v_2} = \emptyset,
\quad \mbox{so by (\ref{def:b}) we may take} \quad b=\max_v |{\cal V}_v(2r)|.
\ena

\ignore{For a regular graph of degree $\delta$, \cite{baldi} give the mean and variance of $Y$ as
\beas
EY=\frac{|{\cal V}|}{\delta+1} \quad \mbox{and} \quad \mbox{Var}(Y)=
\sum_{\alpha,\beta:d(\alpha,\beta)=2}s(\alpha,\beta)(2\delta + 2-s(\alpha,\beta))^{-1}(\delta+1)^{-2},
\enas
where $s(\alpha,\beta)$ is the number of common neighbors of $\alpha$ and $\beta$.}

For example, for $p \in \{1,2,\ldots\}$ and $n \ge 5$ consider the lattice ${\cal V}=\{1,\ldots,n\}^p$ modulo $n$ in $\mathbb{Z}^p$ and
${\cal E}=\{\{v,w\}:d(v,w) = 1\}$; in this case $d$ is the $L^1$ norm
$$
d(v,w)=\sum_{i=1}^p |v_i-w_i|.
$$
Considering the case where we call vertex $v$ a local extreme value if the value $C_v$ exceeds the values $C_w$
over the immediate neighbors $w$ of $v$, we take
$$
{\cal V}_v={\cal V}_v(1) \quad \mbox{and that} \quad |{\cal V}_v(1)| = 1+2p,
$$
the 1 accounting for $v$ itself, and then $2p$ for the number of neighbors at distance 1 from $v$, which differ
from $v$ by either $+1$ or $-1$ in exactly one coordinate.

Lemma \ref{lem:loc}, (\ref{2r-1}), and $|X_v| \le 1$ yield
\bea \label{C:ext}
|Y^s-Y| \le \max_v |{\cal V}_v(2)|=1 + 2p + \left(2p+ 4{p \choose 2}\right)=2p^2+2p+1,
\ena
where the 1 counts $v$ itself, the $2p$ again are the neighbors at distance 1, and the term in the parenthesis accounting for the neighbors at distance 2, $2p$ of them differing in exactly one coordinate by $+2$ or $-2$, and $4 {p \choose 2}$ of them differing by either $+1$ or $-1$ in exactly two coordinates. Note that we have used the assumption $n \ge 5$ here, and continue to do so below.

Now letting $C_v$ have a continuous distribution,
without loss of generality we can assume $C_v \sim {\cal U}[0,1]$. As any vertex has chance $1/|{\cal V}_v|$ of having the largest value in its neighborhood, for the mean $\mu$ of $Y$ we have
\bea \label{mu:ext}
\mu=
\frac{n}{2p+1}.
\ena

To begin the calculation of the variance, note that when $v$ and $w$ are neighbors they cannot both be maxima, so $X_vX_w=0$ and therefore, for $d(v,w)=1$,
\beas
\mbox{Cov}(X_v,X_w)=-(EX_v)^2 = -\frac{1}{(2p+1)^2}.
\enas
If the distance between $v$ and $w$ is 3 or more, $X_v$ and $X_w$ are functions of disjoint sets of independent variables, and hence are independent.

When $d(w,v)=2$ there are two cases, as $v$ and $w$ may have either 1 or 2 neighbors in common, and
\beas
\lefteqn{EX_vX_w=}\\
&&P(U>U_j, V>V_j, j=1,\ldots,m-k \quad \mbox{and} \quad U>U_j,V>U_j, j=m-k+1,\ldots,m),
\enas
where $m$ is the number of vertices over which $v$ and $w$ are extreme, so $m=2p$, and $k=1$ and $k=2$ for the number of neighbors in common.
For $k=1,2,\ldots$, letting  $M_k=\max\{U_{m-k+1},\ldots,U_m\}$, as
the variables $X_v$ and $X_w$ are conditionally independent given $U_{m-k+1},\ldots,U_m$
\bea \nn
E(X_vX_w|U_{m-k+1},\ldots,U_m)&=&P(U>U_j, j=1,\ldots,m|U_{m-k+1},\ldots,U_m)^2\\
 &=&\frac{1}{(m-k+1)^2}(1-M_k^{m-k+1})^2,\label{ext:case2b}
\ena
as
\beas
P(U>U_j, j=1,\ldots,m|U_{m-k+1},\ldots,U_m)&=&\int_{M_k}^1 \int_0^u \cdots \int_0^u du_1 \cdots du_{m-k} du\\
&=&\int_{M_k}^1 u^{m-k}du\\
&=&\frac{1}{m-k+1}(1-M_k^{m-k+1}).
\enas
Since $P(M_k \le x)=x^k$ on $[0,1]$, we have
\beas
EM_k^{m-k+1}&=&k\int_0^1 x^{m-k+1}x^{k-1}dx=\frac{k}{m+1}\quad \mbox{and} \\
E(M_k^{m-k+1})^2&=&k\int_0^1 x^{2(m-k+1)}x^{k-1}dx=\frac{k}{2m-k+2}.
\enas
Hence, averaging (\ref{ext:case2b}) over $U_{m-k+1},\ldots,U_m$ yields
\beas
EX_vX_w=\frac{2}{(m+1)(2(m+1)-k)}.
\enas
For $n \ge 3$, when $m=2p$, for $k=1$ and $2$ we obtain
\beas
\mbox{Cov}(X_v,X_w)=\frac{1}{(2p+1)^2(2(2p+1)-1)} \quad \mbox{and} \quad \mbox{Cov}(X_v,X_w)=\frac{2}{(2p+1)^2(2(2p+1)-2)}, \quad \mbox{respectively.}
\enas
For $n \ge 5$, of
the $2p+4 {p \choose 2}$ vertices $w$ that are at distance 2 from $v$, $2p$ of them share 1 neighbor in common with $v$, while the remaining $4 {p \choose 2}$ of them share 2 neighbors. Hence,
\bea
\nn \sigma^2 &=& \sum_{v \in V} \mbox{Var}(X_v) + \sum_{v \not = w} \mbox{Cov}(X_v,X_w)\\
\nn         &=& \sum_{v \in V} \mbox{Var}(X_v) + \sum_{d(v,w)=1}  \mbox{Cov}(X_v,X_w)+\sum_{d(v,w)=2}\mbox{Cov}(X_v,X_w)\\
\nn         &=& n\left(\frac{2p}{(2p+1)^2}-2p \frac{1}{(2p+1)^2}+2p \frac{1}{(2p+1)^2(2(2p+1)-1)}+ 4{p \choose 2}\frac{2}{(2p+1)^2(2(2p+1)-2)}\right)\\
\nn &=& n\frac{2p}{(2p+1)^2}\left( \frac{1}{(2(2p+1)-1)}+ \frac{2(p-1)}{(2(2p+1)-2)}\right)\\
&=& n \left( \frac{4p^2-p-1}{(2p+1)^2(4p+1)}\right) \label{ext:sig}       .
\ena

We conclude that (\ref{a}) of Theorem \ref{thm:main} holds with $A=C\mu/\sigma^2$ and $B=C/2\sigma$ with $\mu$, $\sigma^2$ and $C$ given by (\ref{mu:ext}), (\ref{ext:sig}) and (\ref{C:ext}), respectively, that is,
\beas
A= \frac{(2p+1)(4p+1)(2p^2+2p+1)}{4p^2-p-1} \quad \mbox{and} \quad B=\frac{2p^2+2p+1}{2\sqrt{n \left( \frac{4p^2-p-1}{(2p+1)^2(4p+1)}\right)}}.
\enas

\subsection{Urn allocation}
\label{subsec:urn}
In the classical urn allocation model $n$ balls are thrown independently into one of $m$ urns, where, for $i=1,\ldots,m$, the probability a ball lands in the $i^{th}$ urn is $p_i$, with $\sum_{i=1}^m p_i=1$. A much studied quantity of interest is the number of nonempty urns, for which Kolmogorov distance bounds
to the normal were obtained in \cite{eng} and \cite{qu}. In \cite{eng}, bounds were obtained for the uniform case where $p_i=1/m$ for all $i=1,\ldots,m$, while the bounds in \cite{qu} hold for the nonuniform case as well. In \cite{pen} the author considers the normal approximation for the number of isolated balls, that is, the number of urns containing exactly one ball, and obtains Kolmogorov distance bounds to the normal. Using the coupling provided in \cite{pen}, we derive right tail inequalities for the number of non-isolated balls, or, equivalently, left tail inequalities for the number of isolated balls.

For $i=1,\ldots,n$ let $X_i$ denote the location of ball $i$, that is, the number of the urn into which ball $i$ lands. The number $Y$ of non-isolated balls is given by
\beas
Y=\sum_{i=1}^n 1(M_i>0)\quad\mbox{where}\quad M_i=-1+\sum_{j=1}^n 1(X_j=X_i).
\enas

We first consider the uniform case. A construction in \cite{pen} produces a coupling of $Y$ to $Y^s$, having the $Y$ size biased distribution, which satisfies $|Y^s-Y|\le 2$.
Given a realization of $\mathbf{X}=\{X_1,X_2,\ldots,X_n\}$, the coupling proceeds by first selecting a ball $I$, uniformly from $\{1,2,\ldots,n\}$, and independently of $\mathbf{X}$. Depending on the outcome of a Bernoulli variable ${\cal B}$, whose distribution depends on the number of balls found in the urn containing $I$, a different ball $J$ will be imported into the urn that contains ball $I$. In some additional detail, let $\mathcal{B}$ be a Bernoulli variable with success probability
$P(\mathcal{B}=1)=\pi_{M_I}$, where
    $$
    \pi_k=\left\{\begin{array}{ll}
    \frac{P(N>k|N>0)-P(N>k)}{P(N=k)(1-k/(n-1))} & \mbox{if $0\le k\le n-2$}\\
    0 &\mbox{if $k=n-1$},
    \end{array}
    \right.
    $$
with $N\sim\mbox{Bin}(1/m,n-1)$. Now let $J$ be uniformly chosen from $\{1,2,\ldots,n\}\setminus\{I\}$, independent of all other variables. Lastly, if ${\cal B}=1$, move ball $J$ into the same urn as $I$.
It is clear that $|Y'-Y| \le 2$, as at most the occupancy of two urns can affected by the movement of a single ball. We also note that if $M_I=0$, which happens when ball $I$ is isolated, $\pi_0=1$, so that $I$ becomes no longer isolated after relocating ball $J$. We refer the reader to \cite{pen} for a full proof that this procedure produces a coupling of $Y$ to a variable with the $Y$ size biased distribution.

For the uniform case, the following explicit formulas for $\mu$ and $\sigma^2$ can be found in Theorem II.1.1 of \cite{kol},
\bea
\mu&=& n\left( 1-\left(1-\frac{1}{m}\right)^{n-1}\right)\quad \mbox{and}\nn\\
\sigma^2 &= &(n-\mu)+\frac{(m-1)n(n-1)}{m}\left(1-\frac{2}{m}\right)^{n-2}-(n-\mu)^2\nn\\
&=&n\left(1-\frac{1}{m}\right)^{n-1}
+\frac{(m-1)n(n-1)}{m}\left(1-\frac{2}{m}\right)^{n-2}-n^2\left(1-\frac{1}{m}\right)^{2n-2}.\label{mu-sig-urn}
\ena
Hence with $\mu$ and $\sigma^2$ as in (\ref{mu-sig-urn}), we can apply (\ref{b}) of Theorem \ref{thm:main} for $Y$, the number of non isolated balls with $C=2$, $A=2\mu/\sigma^2$ and $B=1/\sigma$.

Taking limits in (\ref{mu-sig-urn}), if $m$ and $n$ both go to infinity in such a way that $n/m\rightarrow \alpha\in(0,\infty)$, the mean $\mu$ and variance $\sigma^2$ obey
\beas
\mu \asymp n (1-e^{-\alpha})\quad\mbox{and}\,\, \sigma^2 \asymp n g(\alpha)^2 \,\, \mbox{where} \,\, g(\alpha)^2= e^{-\alpha}-e^{-2\alpha}(\alpha^2-\alpha+1) >0 \,\, \mbox{for all $\alpha \in (0,\infty)$},
\enas
where for positive functions $f$ and $h$ depending on $n$ we write $f \asymp h$ when
$\lim_{n\rightarrow\infty}f/h=1$.

Hence, in this limiting case $A$ and $B$ satisfy
\beas
A \asymp \frac{2(1-e^{-\alpha})}{e^{-\alpha}-e^{-2\alpha}(\alpha^2-\alpha+1)} \quad \mbox{and} \quad B \asymp\frac{1}{\sqrt{n}g(\alpha)}.
\enas
In the nonuniform case similar results hold with some additional conditions. Letting
\beas
||p||=\sup_{1 \le i \le m} p_i\quad\mbox{and}\quad \gamma=\gamma(n)=\max(n||p||,1),
\enas
in \cite{pen} it is shown that when $||p||\le 1/11$ and $n\ge
 83\gamma^2(1+3\gamma+3\gamma^2)e^{1.05\gamma}$,
there exists a coupling such that
\beas
|Y^s-Y|\le 3\quad\mbox{and}\quad\frac{\mu}{\sigma^2}\le 8165\gamma^2 e^{2.1\gamma}.
\enas
Now also using Theorem 2.4 in \cite{pen} for a bound on $\sigma^2$, we find that (\ref{b}) of Theorem \ref{thm:main} holds with
\beas
A=24,495 \, \gamma^2 e^{2.1\gamma} \quad \mbox{and} \quad B= \frac{1.5 \sqrt{7776}\,\,\gamma e^{1.05 \gamma}}{n\sqrt{\sum_{i=1}^m p_i^2}}.
\enas

\subsection{An application to coverage processes}
We consider the following coverage process, and associated coupling, from \cite{golpen}. Given a collection
$\mathcal{U}=\{U_1,U_2,\ldots,U_n\}$ of independent, uniformly distributed points in the $d$ dimensional torus of volume $n$, that is, the cube $C_n=[0,n^{1/d})^d \subset \mathbb{R}^d$
with periodic boundary conditions, let $V$ denote the total volume of the union of the $n$ balls of fixed radius $\rho$ centered at these $n$ points, and $S$ the number of balls isolated at distance $\rho$, that is, those points for which none of the other $n-1$ points lie within distance $\rho$. The random variables $V$ and $S$ are of fundamental interest in stochastic geometry, see \cite{hall} and \cite{pen2}. If $n\rightarrow\infty$ and $\rho$ remains fixed, both $V$ and $S$ satisfy a central limit theorem \cite{hall,moran,penyu}.  The $L^1$ distance of $V$, properly standardized, to the normal is studied in \cite{chanew} using Stein's method.
The quality of
the normal approximation to the distributions of both $V$ and $S$, in the Kolmogorov metric, is studied in \cite{golpen} using Stein's method via size bias couplings.

In more detail, for $x\in C_n$ and $r>0$ let $B_r(x)$ denote the ball of radius $r$ centered at $x$, and $B_{i,r}=B(U_i,r)$. The covered volume $V$ and number of isolated balls $S$ are given, respectively, by
\bea
V =  \mbox{Volume}(\bigcup_{i=1}^{n} B_{i,\rho})\quad \mbox{and} \quad
S= \sum_{i=1}^n \mathbf{1}\{(\mathcal{U}_n \cap B_{i,\rho}=\{U_i\} \}.
\ena
We will derive concentration of measure inequalities for $V$ and $S$ with the help of the bounded size biased couplings
in \cite{golpen}.

Assume $d\ge 1$ and $n\ge 4$. Denote the mean and variance of $V$ by $\mu_V$ and $\sigma_V^2$, respectively, and likewise for $S$, leaving their dependence on $n$ and $\rho$ implicit. Let $\pi_d=\pi^{d/2}/\Gamma(1+d/2)$,
the volume of the unit sphere in $\mathbb{R}^d$, and for fixed $\rho$  let $\phi=\pi_d\rho^d$. For $0\le r\le 2$ let $\omega_d(r)$ denote the volume of the union of two unit balls with centers $r$ units apart.
We have $\omega_1(r)=2+r$, and
\beas
\omega_d(r)=\pi_d+\pi_{d-1}\int_0^r(1-(t/2)^2)^{(d-1)/2}dt, \quad \mbox{for $d \ge 2$.}
\enas
From \cite{golpen}, the means of $V$ and $S$ are given by
\bea \label{cov:mus}
\mu_V = n \left(1-(1 - \phi  /n)^{n}\right) \quad \mbox{and} \quad
\mu_S= n(1 - \phi  /n)^{n-1},
\ena
and their variances by
\bea
\sigma_V^2=  n\int_{B_{2\rho}({\bf 0})}
\left( 1 - \frac{\rho^d \omega_d(|y|/\rho) }{n} \right)^n dy
+ n(n- 2^d\phi) \left( 1 - \frac{ 2\phi}{n} \right)^n
- n^2 ( 1 - \phi/n)^{2n},
\label{cov:varv}
\ena
and
\bea
\sigma_S^2
& = & n (1 -  \phi /n)^{n-1} (1 - (1 - \phi/n)^{n-1})
\nonumber \\
& & + (n-1) \int_{B_{2 \rho} ({\bf 0}) \setminus B_\rho({\bf 0})}
 \left( 1- \frac{\rho^d \omega_d(|y|/\rho)}{n} \right)^{n-2} dy
\nonumber \\
& & + n(n-1) \left( \left(1 - \frac{2^d \phi}{n} \right) \left(1 - \frac{2
\phi}{n} \right)^{n-2} -  \left(1 - \frac{\phi}{n} \right)^{2n-2}
\right).
\ena

It is shown in \cite{golpen}, by using a coupling similar to the one briefly described for the urn allocation problem in Section \ref{subsec:urn}, that one can construct $V^s$ with the $V$ size bias distribution which satisfies $|V^s-V| \le \phi$. Hence (\ref{a}) of  Theorem \ref{thm:main} holds for $V$ with
\beas
A_V=\frac{\phi \mu_V}{\sigma_V^2} \quad \mbox{and} \quad B_V=\frac{\phi}{2 \sigma_V},
\enas
where $\mu_V$ and $\sigma_V^2$ are given in (\ref{cov:mus}) and (\ref{cov:varv}), respectively.
Similarly, with $Y=n-S$ the number of non-isolated balls, it is shown that $Y^s$ with $Y$ size bias distribution can be constructed so that $|Y^s-Y| \le\kappa_d+1$, where $\kappa_d$ denotes the maximum number of open unit balls in $d$ dimensions that can be packed so they all intersect an open unit ball in the origin, but are disjoint from each other.
Hence (\ref{a}) of  Theorem \ref{thm:main} holds for $Y$ with
\beas
A_Y=\frac{(\kappa_d+1)(n- \mu_S)}{\sigma_S^2} \quad \mbox{and} \quad B_Y=\frac{\kappa_d+1}{2 \sigma_S}.
\enas

To see how the $A_V, A_Y$ and $B_V,B_Y$ behave as $n \rightarrow \infty$, let
\beas
J_{r,d}(\rho)=d\pi_d\int_0^r\exp(-\rho^d\omega_d(t))t^{d-1}dt,
\enas
and define
\beas
g_V(\rho)&=&\rho^d J_{2,d}(\rho)-(2^d\phi+\phi^2)e^{-2\phi} \quad \mbox{and} \\
g_S(\rho)&=&e^{-\phi}-(1+(2^d-2)\phi+\phi^2)e^{-2\phi}+\rho^{d}(J_{2,d}(\rho)-J_{1,d}(\rho)).
\enas
Then, again from \cite{golpen},
\beas
\lim_{n\rightarrow\infty} n^{-1}\mu_V=\lim_{n\rightarrow\infty} (1-n^{-1}\mu_S)&=& 1-e^{-\phi},\\
\lim_{n\rightarrow\infty} n^{-1}\sigma^2_V&=&g_V(\rho)>0, \quad \mbox{and}\\
\lim_{n\rightarrow\infty} n^{-1}\sigma^2_S&=&g_S(\rho)>0.
\enas
Hence, $B_V$ and $B_Y$ tend to zero at rate $n^{-1/2}$, and
\beas
\lim_{n \rightarrow \infty}A_V=\frac{\phi (1-e^{-\phi})}{g_V(\rho)}, \quad \mbox{and} \quad
\lim_{n \rightarrow \infty}A_Y=\frac{(\kappa_d+1)(1- e^{-\phi})}{g_S(\rho)}.
\enas

\subsection{The lightbulb problem}
The following stochastic process, known informally as the `lightbulb process', arises in a pharmaceutical study of dermal patches, see \cite{rrz}. Changing dermal receptors to lightbulbs allows for a more colorful description. Consider $n$ lightbulbs, each operated by a switch. At day zero, none of the bulbs are on. At day $r$ for $r=1,\ldots,n$, the position of $r$ of the $n$ switches are selected uniformly to be changed, independent of the past. One is interested in studying the distribution of the number of lightbulbs which are switched on at the terminal time $n$. The process just described is Markovian, and is studied in some detail in \cite{zl}. In \cite{golzhang} the authors use Stein's method to derive a bound to the normal via a monotone, bounded size bias coupling. Borrowing this coupling here allows for the application of Theorem \ref{thm:main} to obtain concentration of measure inequalities for the lightbulb problem. We begin with a more detailed description of the process.

For $r=1,\ldots,n$, let $\{X_{rk}, k =1, \ldots, n\}$ have distribution
\beas
P(X_{r1}=e_1,\ldots,X_{rn}=e_n)={n \choose r}^{-1} \quad \mbox{for all $e_k \in \{0,1\}$ with $\sum_{k=1}^n e_k=r$,}
\enas
and let these collections of variables be independent over $r$. These `switch variables' $X_{rk}$ indicate whether or not on day $r$ bulb $k$ had its status changed. With
\beas
Y_k= \left( \sum_{r=1}^n X_{rk}\right) \mbox{mod}\,2
\enas
therefore indicating the status of bulb $k$ at time $n$, the number of bulbs switched on at the terminal time is
\beas
Y=\sum_{k=1}^n Y_k.
\enas

From \cite{rrz}, the mean $\mu$ and variance $\sigma^2$ of $Y$ are given by
\bea
\label{def:mu}
\mu = \frac{n}{2}
\left(1- \prod_{i=1}^n \left(1 - \frac{2i}{n}\right) \right),
\ena
and
\bea \label{def:sigma2}
\sigma^2 = \frac{n}{4}\left[1- \prod_{i=1}^n \left(1 -
\frac{4i}{n}+\frac{4i(i-1)}{n(n-1)}\right)\right] +\frac{n^2}{4} \left[\prod_{i=1}^n \left(1 -
\frac{4i}{n}+\frac{4i(i-1)}{n(n-1)}\right) - \prod_{i=1}^n \left(1
- \frac{2i}{n}\right)^2 \right].
\ena
Note that when $n$ is even $\mu=n/2$ exactly, as the product in (\ref{def:mu}) is zero, containing the term $i=n/2$. By results in \cite{rrz}, in the odd case $\mu=(n/2)(1+O(e^{-n}))$, and in both the even and odd cases $\sigma^2=(n/4)(1+O(e^{-n}))$.

The following construction, given in \cite{golzhang} for the case where $n$ is even, couples $Y$  to a variable $Y^s$ having the $Y$ size bias distribution such that
\bea \label{2bulb}
Y \le Y^s \le Y+2,
\ena
that is, the coupling is monotone, with difference bounded by 2.
For every $i \in \{1,\ldots,n\}$ construct the collection of variables ${\bf Y}^i$ from ${\bf Y}$ as follows.
If $Y_i=1$, that is, if bulb $i$ is on, let ${\bf Y}^i={\bf Y}$.
Otherwise, with $ J^i={\cal U}\{j: Y_{n/2,j} = 1- Y_{n/2,i}\}$,
let ${\bf Y}^i=\{Y_{rk}^i:r,k=1,\ldots,n\}$ where
\beas
Y_{rk}^i=\left\{
\begin{array}{cl}
Y_{rk} & r \not = n/2\\
Y_{n/2,k}& r=n/2, k \not \in \{i,J^i\}\\
Y_{n/2,J^i} & r=n/2, k=i\\
Y_{n/2,i}& r=n/2, k=J^i,
\end{array}
\right.
\enas
and let $Y^i=\sum_{k=1}^n Y_k^i$ where
$$
Y_k^i = \left(\sum_{r=1}^n Y_{rk}^i\right) \mbox{ mod }2.
$$
Then, with $I$ uniformly
chosen from $\{1,\ldots,n\}$ and independent of all other variables, it is shown in \cite{golzhang} that the mixture $Y^s=Y^I$ has the $Y$ size
biased distribution, essentially due to the fact that
\beas
{\cal L}({\bf Y}^i)={\cal L}({\bf Y}|Y_i=1) \quad \mbox{for all $i=1,\ldots,n$.}
\enas
It is not difficult to see that $Y^s$ satisfies (\ref{2bulb}). If $Y_I=1$ then ${\bf X}^I={\bf X}$, and so in this case $Y^s=Y$. Otherwise $Y_I=0$, and for the given $I$ the collection ${\bf Y}^I$ is constructed from ${\bf Y}$ by interchanging the stage $n/2$, unequal, switch variables $Y_{n/2,I}$ and $Y_{n/2,J^I}$. If $Y_{J^I}=1$ then after the interchange $Y_I'=1$ and $Y_{J^I}'=0$, in which case $Y^s=Y$. If $Y_{J^I}=0$ then after the interchange $Y_I^I=1$ and $Y_{J^I}^I=1$, yielding $Y^s=Y+2$. We conclude that for the case $n$ even $C=2$ and
(\ref{a}) and (\ref{b}) of Theorem \ref{thm:main} hold with
\bea \label{ABbulb-even}
A=n/\sigma^2 \quad \mbox{and} \quad B=1/\sigma
\ena
where $\sigma^2$ is given by (\ref{def:sigma2}).

For the coupling in the odd case, $n=2m+1$ say, due to the parity issue, \cite{golzhang} considers a random variable $V$ close to $Y$ constructed as follows. In all stages but stage $m$ and $m+1$ let the switch variables which will yield $V$ be the same as those for $Y$. In stage $m$, however, with probability $1/2$ one applies an additional switch variable, and independently in stage $m+1$, with probability $1/2$, one switch variable fewer. In this way the switch variables in these two stages have the same, symmetric distribution and are close to the switch variables for $Y$. In particular, as at most two switch variables are different in the configuration for $V$, we have $|V-Y| \le 2$. Helped by the symmetry, one may couple $V$ to a variable $V^s$ with the $V$ size bias distribution as in the even case, obtaining $V \le V^s \le V+2$. Hence (\ref{a}) and (\ref{b}) of Theorem \ref{thm:main} hold for $V$ as for the even case with values given in (\ref{ABbulb-even}), where
$\mu=n/2$ and $\sigma^2=(n/4)(1+O(e^{-n})$. Since $|V-Y| \le 2$, by replacing $t$ by
$t+2/\sigma$ in the bounds for $V$ one obtains bounds for the odd case $Y$.

\section{Applications: unbounded couplings}
\label{sec:unbounded}

One of the major drawbacks of Theorem \ref{thm:main} is the hypothesis that $|Y^s-Y|$ be almost surely bounded with probability one. In \cite{graph}, ideas similar to the previous sections are applied to obtain subgaussian concentration of measure inequalities for the number of isolated vertices in the Erd\H{o}s-R\'{e}nyi random graph model, employing a coupling that does not obey the boundedness condition. In this section we derive concentration of measure inequalities for another example where $Y^s-Y$ is not bounded: the nonnegative infinitely divisible distributions with certain associated moment generating functions which satisfy a boundedness condition. As an example for nonnegative infinitely divisible distribution, compound Poisson distributions will be our main illustration.

\subsubsection{Infinitely divisible distributions}
When $Y$ is Poisson then $Y^s=Y+1$ and we may write
\bea \label{decomp-poi}
Y^s =Y+X
\ena
with $X$ and $Y$ independent. Theorem 5.3 of \cite{ste} shows that if $Y$ is nonnegative with finite mean then (\ref{decomp-poi}) holds if and only if $Y$ is infinitely divisible.
Hence, in this case, a coupling of $Y$ to $Y^s$ may be achieved by generating the independent variable $X$ and adding it to $Y$. Since $Y^s$ is always stochastically larger than $Y$ we must have $X \ge 0$, and therefore this coupling is monotone. In addition $Y^s-Y=X$ so the coupling is bounded if and only if $X$ is bounded.
When $X$ is unbounded, Theorem \ref{thm-comp-poi} provides concentration of measure inequalities for $Y$
under appropriate growth conditions on two generating functions in $Y$ and $X$.
We assume without further mention that $Y$ is nontrivial, and note that therefore the means of both $Y$ and $X$ are positive.

\begin{theorem}\label{thm-comp-poi}
Let $Y$ have a nonnegative infinitely divisible distribution and suppose that there exists $\gamma>0$ so that $E(e^{\gamma Y})<\infty$. Let $X$ have the distribution such that (\ref{decomp-poi}) holds when $Y$ and $X$ are independent, and assume $E(Xe^{\gamma X})= C<\infty$. Letting $\mu=E(Y), \sigma^2=\mbox{Var}(Y),\nu=E(X)$ and $K=(C+\nu)/2$, the following concentration of measure inequalities hold for all $t>0$,
\beas
P\left(\frac{Y-\mu}{\sigma}\ge t\right)\le
\left\{\begin{array}{ll} \exp\left(-\frac{t^2\sigma^2}{2K\mu}\right) & \mbox{for $t \in [0,\gamma K\mu/\sigma^2)$}\\
\exp\left(-\gamma t+\frac{K\mu\gamma^2}{2\sigma^2}\right)&\mbox{for $t \in [\gamma K\mu/\sigma^2,\infty)$,}
\end{array}\right.
\quad \mbox{and} \,\,
P\left(\frac{Y-\mu}{\sigma}\le -t\right) \le \exp\left(-\frac{t^2\sigma^2}{2\nu\mu}\right).
\enas
\end{theorem}

\begin{proof}
 Since $Y^s=Y+X$ with $Y$ and $X$ independent and
$X \ge 0$, using (\ref{ineq-main}) with
$\theta \in (0,\gamma)$ we have,
\beas
E(e^{\theta Y^s}-e^{\theta Y})&=&E(e^{\theta (X+Y)}-e^{\theta Y})\le \frac{1}{2}E\left(\theta X(e^{\theta (X+Y)}+e^{\theta Y})\right)\\
&=&\frac{\theta}{2}E\left( X(e^{\theta X}+1)e^{\theta Y}\right)=\frac{\theta}{2}E\left(X(e^{\theta X}+1)\right) E(e^{\theta Y})\\
&\le & \frac{\theta}{2}(E(Xe^{\gamma X})+E(X))E(e^{\theta Y})\\
&=& K \theta m(\theta) \quad \mbox{where $K=(C+\nu)/2$ and $m(\theta)=E(e^{\theta Y})$.}
\enas
Now adding $m(\theta)$ to both sides yields
\beas
E(e^{\theta Y^s})\le (1+K\theta)m(\theta),
\enas
and therefore
\bea\label{bd-diff-comp}
m'(\theta)=E(Ye^{\theta Y})=\mu E(e^{\theta Y^s})\le \mu (1+K\theta) m(\theta).
\ena
Again, with $M(\theta)$ the moment generating function of $(Y-\mu)/\sigma$,
\beas
M(\theta)=Ee^{\theta(Y-\mu)/\sigma}=e^{-\theta \mu/\sigma}m(\theta/\sigma),
\enas
by (\ref{bd-diff-comp}) we have,
\bea
M'(\theta) &=& -(\mu/\sigma) e^{-\theta \mu/\sigma}m(\theta/\sigma) +  e^{-\theta \mu/\sigma}m'(\theta/\sigma)/\sigma\nonumber\\
           &\le& -(\mu/\sigma) e^{-\theta \mu/ \sigma}m(\theta/\sigma) + (\mu/\sigma) e^{-\theta \mu/\sigma}\left( 1+K\frac{\theta}{\sigma} \right) m(\theta/\sigma)\nonumber\\
           &=& (\mu/\sigma^2) K\theta M(\theta).\label{bd-M-prime-comp}
\ena
Integrating, and using the fact that $M(0)=1$ yields
\beas
M(\theta)\le \exp\left(\frac{K\mu\theta^2}{2\sigma^2}\right)\quad\mbox{for $\theta \in (0,\gamma)$.}
\enas
Hence for a fixed $t>0$, for all $\theta \in (0,\gamma)$,
\beas
P\left(\frac{Y-\mu}{\sigma}\ge t\right)\le e^{-\theta t}M(\theta)\le \exp\left(-\theta t+\frac{K\mu\theta^2}{2\sigma^2}\right).
\enas
The infimum of the quadratic in the exponent is attained at $\theta=t\sigma^2/K\mu$. When this value
lies in $(0,\gamma)$ we obtain the first, right tail bound, for
$t$ in the bounded interval, while setting $\theta=\gamma$ yields the second.

Moving on to the left tail bound, using (\ref{ineq-main}) for $\theta<0$ yields
\beas
E(e^{\theta Y}-e^{\theta Y^s})&\le & -\frac{\theta}{2}E((Y^s-Y)(e^{\theta Y}+e^{\theta Y^s}))
\le  -\theta E(Xe^{\theta Y})=-\theta E(X)E(e^{\theta Y}).
\enas
Rearranging we obtain
\beas
m'(\theta)=\mu E(e^{\theta Y^s})\ge \mu(1+\theta\nu)m(\theta).
\enas

Following calculations similar to (\ref{bd-M-prime-comp}) one obtains
\beas
M'(\theta)\ge (\mu/\sigma^2)\nu \theta  M(\theta) \quad \mbox{for all $\theta<0$,}
\enas
which upon integration over $[\theta,0]$ yields
\beas
M(\theta)\le \exp\left(\frac{\nu\mu\theta^2}{2\sigma^2}\right)\quad\mbox{for all $\theta<0$}.
\enas
Hence for any fixed $t>0$, for all $\theta<0$,
\bea
P\left(\frac{Y-\mu}{\sigma}\le -t\right)\le e^{\theta t} M(\theta)\le \exp\left(\theta t+\frac{\nu\mu\theta^2}{2\sigma^2}\right).\label{left-comp-pf}
\ena
Substituting $\theta=-t\sigma^2/(\nu\mu)$ in (\ref{left-comp-pf}) yields the lower tail bound,
thus completing the proof.
\end{proof}

Though Theorem \ref{thm-comp-poi} applies in principle to all nonnegative infinitely divisible distributions with generating functions for $Y$ and $X$ that satisfy the given growth conditions, we now specialize to the subclass of compound Poisson distributions, over which it is always possible to determine the independent increment $X$. Not too much is sacrificed in narrowing the focus to this case, since a nonnegative infinitely divisible random variable $Y$ has a compound Poisson distribution if and only if $P(Y=0)>0$.

\subsubsection{Compound Poisson distribution}
One important subfamily of the infinitely divisible distributions are the compound Poisson distributions, that is, those distributions that are given by
\bea \label{YisNZi}
Y=\sum_{i=1}^N Z_i,\quad\mbox{where $N\sim\mbox{Poisson}(\lambda)$, and $\{Z_i\}_{i=1}^\infty$ are independent and distributed as $Z$.}
\ena
Compound Poisson distributions are popular in several applications, such as insurance mathematics, seismological data modelling, and reliability theory; the reader is referred to \cite{bar} for a detailed review.

Although $Z$ is not in general required to be nonnegative, in order to be able to size bias $Y$ we restrict ourselves to this situation.  It is straightforward to verify that when the moment generating function $m_Z(\theta)=Ee^{\theta Z}$ of $Z$ is finite, then the moment generating function
$m(\theta)$ of $Y$ is given by
\beas
m(\theta)=\exp(-\lambda(1-m_Z(\theta))).
\enas
In particular $m(\theta)$ is finite whenever $m_Z(\theta)$ is finite. As $Y$ in (\ref{YisNZi}) is infinitely divisible the equality (\ref{decomp-poi}) holds for some $X$; the following lemma determines the distribution of $X$ in this particular case.

\begin{lemma}
\label{lemma-ag}
Let $Y$ have the compound Poisson distribution as in (\ref{YisNZi}) where $Z$ is nonnegative and has finite, positive mean.
Then
\beas
Y^s=Y+Z^s,
\enas
has the $Y$ size biased distribution, where $Z^s$ has the $Z$ size bias distribution and is independent of $N$ and $\{Z_i\}_{i=1}^\infty$.
\end{lemma}

\begin{proof}
Let $\phi_V(u)=Ee^{iuV}$ for any random variable $V$. If $V$ is nonnegative
and has finite positive mean, using $f(y)=e^{iuy}$ in (\ref{EWfWchar}) results in
\bea \label{chVs}
\phi_{V^s}(u)=\frac{1}{EV}\left( EVEe^{iuV^s} \right)=\frac{1}{EV}EVe^{iuV}=\frac{1}{iEV}\phi_V'(u).
\ena
It is easy to check that the characteristic function of the compound Poisson $Y$ in (\ref{YisNZi})
is given by
\bea\label{cf-comp}
\phi_Y(u)=\exp(-\lambda(1-\phi_Z(u))),
\ena
and letting $EZ=\vartheta$, that $EY=\lambda \vartheta$.
Now applying (\ref{chVs}) and (\ref{cf-comp}) results in
\beas
\phi_{Y^s}(u)=\frac{1}{i\lambda \vartheta}\phi_Y'(u)=\frac{1}{i \vartheta}\phi_Y(u)\phi_Z'(u)
= \phi_Y(u)\phi_{Z^s}(u).
\enas
\end{proof}

To illustrate Lemma \ref{lemma-ag}, consider the Cram\'{e}r-Lundberg model \cite{emb} from insurance mathematics. Suppose an insurance company starts with an initial capital $u_0$, and premium is collected at the constant rate $\alpha$. Claims arrive according to a homogenous Poisson process $\{N_\tau\}_{\tau \ge 0}$ with rate $\lambda$, and the claim sizes are independent with common distribution $Z$. The aggregate claims $Y_\tau$ made by time $\tau \ge 0$ is therefore given by (\ref{YisNZi}) with $N$ and $\lambda$ replaced by $N_\tau$ and $\lambda_\tau$, respectively.

Distributions for $Z$ which are of interest for applications include the Gamma, Weibull, and Pareto, among others.
For concreteness, if $Z \sim \mbox{Gamma}(\alpha,\beta)$ then $Z^s \sim \mbox{Gamma}(\alpha+1,\beta)$,
and the mean $\nu$ of the increment $Z^s$, and the mean $\mu_\tau$ and variance $\sigma_\tau^2$ of $Y_\tau$, are given by
$$
\nu= \quad (\alpha+1)\beta, \quad \mu_\tau=  \lambda \tau \alpha \beta \quad\mbox{and}\quad \sigma_\tau^2=\lambda \tau \beta^2\alpha.
$$

The conditions of Theorem \ref{thm-comp-poi} are satisfied
with any $\gamma \in (0,1/\beta)$ since $E(e^{\theta Y})<\infty$ and $E(Z^se^{\theta Z^s})<\infty$ for all $\theta<1/\beta$. Taking $\gamma=1/(M\beta)$ for $M >1$ for example, yields
$$
C=E(Z^se^{\gamma Z^s})=(\alpha+1)\beta (\frac{M}{M-1})^{\alpha+2}.
$$
For instance, the lower tail bound of Theorem \ref{thm-comp-poi} now yields a bound on the probability that the aggregate claims by time $\tau$ will be `small', of
\beas
P\left(\frac{Y_\tau-\mu_\tau}{\sigma_\tau}\le -t \right)\le \exp\left(-\frac{t^2 }{2(\alpha+1)}\right).
\enas
It should be noted that in some applications one may be interested in $Z$ which are heavy tailed, and hence do not satisfy the conditions in Theorem \ref{thm-comp-poi}.


\begin{thebibliography}{99}
\bibitem{baldi} \textsc{Baldi, P.}, \textsc{Rinott, Y.} and \textsc{Stein, C.} (1989). A normal approximations for the number of local maxima of a random function on a graph, \textit{Probability, Statistics and Mathematics, Papers
in Honor of Samuel Karlin}, T. W. Anderson, K.B. Athreya and D. L. Iglehart eds., Academic
Press, 59-81.

\bibitem{ChBa}\textsc{Barbour, A.D.} and \textsc{Chen, L.H.Y}(2005). {An Introduction to Stein's Method}, Chen,L.H.Y and Barbour,A.D. eds,Lecture Notes Series No. 4, Institute for Mathematical Sciences, National University of Singapore, Singapore University Press and World Scientific 2005, 1-59.

\bibitem{bar}\textsc{Barbour, A.D.} and \textsc{Chryssaphinou, O.}(2001). Compound Poisson approximation: A user's guide, \textit{Ann. Appl. Probab.}, \textbf{11}, 964-1002.



\bibitem{BaDia} \textsc{Bayer, D.} and \textsc{Diaconis, P.}(1992). Trailing the Dovetail Shuffle to its
Lair. {\em Ann. Appl. Probab.} {\bf 2}, 294-313.

\bibitem{cha} \textsc{Chatterjee, S.}(2007). {Stein's method for concentration inequalities}, \textit{Probab. Th. Rel. Fields}, \textbf{138}, 305-321. \href{http://arxiv.org/abs/math/0604352}{arXiv 0604352}


\bibitem{CGS}\textsc{Chen, L.H.Y., Goldstein, L. and Shao, Q.M} (2011) Normal Approximation by Stein's Method. Springer.

\bibitem{chanew} \textsc{Chatterjee, S.}(2008){A new method of normal approximation}, \textit{Ann. Probab.},
\textbf{4}, 1584-1610.   \href{http://arxiv.org/abs/math/0611213}{arXiv 0611213}

\bibitem{emb}\textsc{Embrechts, P.} and \textsc{Kl\"{u}ppelberg, C.}(1993). Some aspects of insurance mathematics,  \textit{Th. Probab. Appl.}, \textbf{38}, 262-295.

\bibitem{eng}\textsc{Englund, G.}(1981). A remainder term estimate for the normal approximation in classical occupancy, \textit{Ann. Probab.}, \textbf{9}, 684-692.

\bibitem{feller2} \textsc{Feller, W.}(1966). An Introduction to Probability and
its Applications, volume II. Wiley.

\bibitem{graph}\textsc{Ghosh, S.}, \textsc{Goldstein, L.} and \textsc{Rai\v{c}, M.}(2011). {Concentration of measure for the number of isolated vertices in the Erd\H{o}s-R\'{e}nyi random graph by size bias couplings}, \textit{Stat. and Probab. Letters}, to appear. \href{http://arxiv.org/abs/1106.0048}{arXiv 1106.0048}

\bibitem{goldstein1}\textsc{Goldstein, L.}(2005). {Berry Esseen Bounds for
Combinatorial Central Limit Theorems and Pattern Occurrences,
using Zero and Size Biasing}, \textit{J. Appl. Probab.}, \textbf{42}, 661-683.  \href{http://arxiv.org/abs/math/0511510}{arXiv 0511510}

\bibitem{golpen}\textsc{Goldstein, L.} and \textsc{Penrose, M.}(2010). {Normal approximation for coverage models over binomial point processes}, \textit{Ann. Appl. Probab.},  {\bf 20}, 696-721. \href{http://arxiv.org/abs/0812.3084}{arXiv 0812.3084}

\bibitem{golrinott}\textsc{Goldstein, L.} and \textsc{Rinott, Y.}(1996). {Multivariate normal approximations by Stein's method and size bias couplings}, \textit{J. Appl. Probab.}, \textbf{33},1-17. \href{http://arxiv.org/abs/math/0510586}{arXiv 0510586}

\bibitem{golzhang} \textsc{Goldstein, L.} and \textsc{Zhang, H.} (2009). {A Berry-Esseen theorem for the lightbulb process},  \textit{Adv. Appl. Probab., to appear}. \href{http://arxiv.org/abs/1001.0612}{arXiv 1001.0612}

\bibitem{hall}\textsc{Hall, P.}(1988). Introduction to the theory of coverage processes, John Wiley, New York.

\bibitem{kol} \textsc{Kolchin, V.F.}, \textsc{Sevast'yanov, B.A.} and \textsc{Chistyakov, V.P.}(1978). {Random Allocations}, Winston, Washington D.C.


\bibitem{ledoux}\textsc{Ledoux, M.}(2001). {The concentration of measure phenomenon}, Amer. Math. Soc., Providence, RI.

\bibitem{midzuno} \textsc{Midzuno, H.} (1951). On the sampling system with probability proportionate to sum of sizes, \textit{Ann. Inst. Stat. Math.}, \textbf{2}, 99-108.

\bibitem{moran}\textsc{Moran, P.A.P.}(1973). {The random volume of interpenetrating spheres in space}, \textit{J. Appl. Probab.}, \textbf{10}, 483-490.


\bibitem{pen2} \textsc{Penrose, M.}(2003). Random geometric graphs, Oxford University Press, Oxford.

\bibitem{pen}\textsc{Penrose, M.}(2009). {Normal approximation for isolated balls in an urn allocation model},  \href{http://arxiv.org/abs/0901.3493}{arXiv 0901.3493}

\bibitem{penyu}\textsc{Penrose, M.D.} and \textsc{Yukich, J.E.}(2001). Central limit theorems for some graphs in computational geometry, \textit{Ann. Appl. Probab.}, \textbf{11}, 1005-1041.

\bibitem{qu}\textsc{Quine, M.P.} and \textsc{Robinson, J.}(1982). A Berry Esseen bound for an occupancy problem, \textit{Ann. Probab}, \textbf{10}, 663-671.

\bibitem{raic}\textsc{Rai\v{c}, M.}(2007). CLT related large deviation bounds based on Stein's method, \textit{Adv. Appl. Prob.}, \textbf{39}, 731-752.

\bibitem{rrz} \textsc{Rao, C.R.}, \textsc{Rao, B.M.}, and \textsc{Zhang, H.}(2007). {One Bulb? Two Bulbs? How Many Bulbs Light Up? A Discrete Probability Problem Involving Dermal Patches}, \textit{Sankhy$\bar{a}$}, {\bf 69}, 137-161.

\bibitem{multi} \textsc{Reinert, G.} and \textsc{R\"{o}llin, A.}(2009). {Multivariate normal approximation with Stein's method of exchangeable pairs under a general linearity condition}, \textit{Ann. Probab.}, {\bf 37},  2150-2173. \href{http://arxiv.org/abs/0711.1082}{arXiv 0711.1082}

\bibitem{stein} \textsc{Stein, C.} (1972). {A bound for the error in the normal approximation to the distribution
of a sum of dependent random variables}, \textit{Proc. Sixth Berkeley Symp. Math. Statist.
Probab.} \textbf{2}, 583-602, Univ. California Press, Berkeley.

\bibitem{Stein86} \textsc{Stein, C.} (1986). Approximate Computation of Expectations. Institute of Mathematical Statistics, Hayward, CA.

\bibitem{ste}\textsc{Steutel, W.F.}(1973). Some recent results in infinite divisibility. \textit{Stoch. Proc. Appl.}, \textbf{1}, 125-143.

\bibitem{zl} \textsc{Zhou, H.} and \textsc{Lange, K.} (2009). {Composition Markov chains of multinomial type.} \textit{Adv. Appl. Probab.}, {\bf 41}, 270-291.

\end{thebibliography}
\end{document}